\documentclass[a4paper,12pt]{article}
\usepackage{latexsym,amsmath,amsthm,amssymb}
\usepackage{a4wide}
\usepackage{hyperref}
\usepackage{marginnote}
\usepackage{color}
\usepackage{xcolor}
\usepackage{cite}
\hypersetup{
pdftitle={Adams hyperbolic}   
pdfauthor={Van Hoang Nguyen},
colorlinks = true,
linkcolor = magenta,
citecolor = blue,
}

\theoremstyle{plain}
\newtheorem{theorem}{Theorem}[section]

\newtheorem{proposition}[theorem]{Proposition}
\newtheorem{lemma}[theorem]{Lemma}

\theoremstyle{definition}

\theoremstyle{remark}

\renewcommand{\thefootnote}{\arabic{footnote}}

\def\R{\mathbb R}

\def\si{\sigma}
\def\vphi{\varphi}
\def\ep{\epsilon}
\def\na{\nabla}
\def\la{\langle} 
\def\ra{\rangle} 
\def\lt{\left}
\def\rt{\right}

\def\H{\mathbb H}

\def\leqs{<}
\def\geqs{>}

\numberwithin{equation}{section}

\title{The sharp higher order Lorentz--Poincar\'e and Lorentz--Sobolev inequalities in the hyperbolic spaces}
\author{Van Hoang Nguyen
}

\begin{document}
\maketitle


\renewcommand{\thefootnote}{}

\footnote{Email:  \href{mailto: Van Hoang Nguyen <vanhoang0610@yahoo.com>}{vanhoang0610@yahoo.com}.}

\footnote{2010 \emph{Mathematics Subject Classification\text}: 26D10, 46E35, 46E30, }

\footnote{\emph{Key words and phrases\text}: Poincar\'e inequality, Poincar\'e--Sobolev inequality, Lorentz--Sobolev space, hyperbolic spaces.}

\renewcommand{\thefootnote}{\arabic{footnote}}
\setcounter{footnote}{0}

\begin{abstract}
In this paper, we study the sharp Poincar\'e inequality and the Sobolev inequalities in the higher order Lorentz--Sobolev spaces in the hyperbolic spaces. These results generalize the ones obtained in \cite{Nguyen2020a} to the higher order derivatives and seem to be new in the context of the Lorentz--Sobolev spaces defined in the hyperbolic spaces.

\end{abstract}

\section{Introduction}
For $n\geq 2$, let us denote by $\H^n$ the hyperbolic space of dimension $n$, i.e., a complete, simply connected, $n-$dimensional Riemmanian manifold having constant sectional curvature $-1$. The aim in this paper is to generalize the main results obtained by the author in \cite{Nguyen2020a} to the higher order Lorentz--Sobolev spaces in $\H^n$. Before stating our results, let us fix some notation. Let $V_g, \na_g$ and $\Delta_g$ denote the volume element, the hyperbolic gradient and the Laplace--Beltrami operator in $\H^n$ with respect to the metric $g$ respectively. For higher order derivatives, we shall adopt the following convention
\[
\na_g^m \cdot = \begin{cases}
\Delta_g^{\frac m2} &\mbox{if $m$ is even,}\\
\na_g (\Delta_g^{\frac{m-1}2} \cdot) &\mbox{if $m$ is odd.}
\end{cases}
\]
Furthermore, for simplicity, we write $|\na^m_g \cdot|$ instead of $|\na_g^m \cdot|_g$ when $m$ is odd if no confusion occurs. For $1\leq p, q\leqs \infty$, we denote by $L^{p,q}(\H^n)$ the Lorentz space in $\H^n$ and by $\|\cdot\|_{p,q}$ the Lorentz quasi-norm in $L^{p,q}(\H^n)$. When $p=q$, $\|\cdot\|_{p,p}$ is replaced by $\|\cdot\|_p$ the Lebesgue $L_p-$norm in $\H^n$, i.e., $\|f\|_p = (\int_{\H^n} |f|^p dV_g)^{\frac1p}$ for a measurable function $f$ on $\H^n$. The Lorentz--Sobolev space $W^m L^{p,q}(\H^n)$ is defined as the completion of $C_0^\infty(\H^n)$ under the Lorentz quasi-norm $\|\na_g^m u\|_{p,q}:=\| |\na_g^m u| \|_{p,q}$. In \cite{Nguyen2020a}, the author proved the following Poincar\'e inequality in $W^1 L^{p,q}(\H^n)$
\begin{equation}\label{eq:Poincare1}
\|\na_g u\|_{p,q}^q \geq \lt(\frac{n-1}p\rt)^q \|u\|_{p,q}^q,\quad\forall\, u\in W^1 L^{p,q}(\H^n).
\end{equation}
provided $1\leqs q \leq p$. Furthermore, the constant $(\frac{n-1}p)^q$ in \eqref{eq:Poincare1} is the best possible and is never attained. The inequality \eqref{eq:Poincare1} generalizes the result in \cite{NgoNguyenAMV} to the setting of Lorentz--Sobolev space. The first main result in this paper extends the inequality \eqref{eq:Poincare1} to the higher order Sobolev space $W^{m} L^{p,q}(\H^n)$.

\begin{theorem}\label{MAINI}
Given $n \geq 2$, $m\geq 1$ and $1 \leqs p \leqs \infty$, let us denote the following constant
\begin{equation*}
C(n,m,p) = \begin{cases}
(\frac{(n-1)^2}{pp'})^{\frac m2} &\mbox{if $m$ is even,}\\
\frac {n-1}p (\frac{(n-1)^2}{pp'})^{\frac {m-1}2}&\mbox{if $m$ is odd,}
\end{cases}
\end{equation*}
where $p' = \frac p{p-1}$. Then the following Poincar\'e inequality holds in $W^{m} L^{p,q}(\H^n)$
\begin{equation}\label{eq:Poincarehigher}
\|\na_g^m u\|_{p,q}^q \geq C(n,m,p)^q \|u\|_{p,q}^q,\quad u\in W^m L^{p,q}(\H^n)
\end{equation}
for any $1\leqs p,q \leqs \infty$ if $m$ is even, or for any $1\leqs q \leq p\leqs \infty$ if $m$ is odd. Moreover, the constant $C(n,m,p)$ in \eqref{eq:Poincarehigher} is sharp and is never attained.
\end{theorem}
Let us give some comments on Theorem \ref{MAINI}.  The Poincar\'e inequality in the hyperbolic space was proved by Tataru \cite{Tataru2001}
\begin{equation}\label{eq:Tataru}
\int_{\H^n} |\na_g^m u|^p dV_g \geq C \int_{\H^n} |u|^p dV_g,\quad u\in C_0^\infty(\H^n)
\end{equation}
for some constant $C \geqs 0$. The sharp value of constant $C$ in \eqref{eq:Tataru} is computed by Mancini and Sandeep \cite{ManciniSandeep2008} when $p=2$ and by Ngo and the author \cite{NgoNguyenAMV} for arbitrary $p$ (see \cite{BerchioNA} for another proof when $m=1$). Theorem \ref{MAINI} gives an extension of the Poincar\'e inequality \eqref{eq:Tataru} with the sharp constant to the higher order Sobolev spaces $W^m L^{p,q}(\H^n)$. Similar to the case $m=1$ established in \cite{Nguyen2020a}, we need an extra condition $q \leq p$ when $m$ is odd to apply the symmetrization argument. The proof of Theorem \ref{MAINI} follows the idea in the proof of Theorem $1.1$ in \cite{NgoNguyenAMV} by using the iterate argument. The main step in the proof is to establish the inequality when $m=2$. The case $m=1$ was already done in \cite{Nguyen2020a}.

There have been many improvements of \eqref{eq:Tataru} with the sharp constant in literature. For examples, the interesting readers may consult the papers \cite{ManciniSandeep2008,BenguriaFrankLoss,BerchioNA,BerchioJFA,NguyenPS2018} for the improvements of \eqref{eq:Tataru} for $m=1$ by adding the remainder terms concerning to Hardy weights or to the $L^q-$norms with $p \leqs q \leq \frac{np}{n-p}$. For the higher order Sobolev spaces, we refer the readers to the papers of Lu and Yang \cite{LuYang2019,LuYang2019a,Hong2019,Nguyenorder2}. Especially, in \cite[Theorem $1.1$]{Nguyenorder2} the author established the following improvement of \eqref{eq:Tataru} for $p=2$
\begin{equation}\label{eq:NguyenPSorder2}
\int_{\H^n} |\Delta_g u|^2 dV_g -\frac{(n-1)^4}{16} \int_{\H^n} |u|^2 dV_g \geq S_{n,2} \lt(\int_{\H^n} |u|^{\frac{2n}{n-4}} dV_g\rt)^{\frac{n-4}n}, \quad u\in C_0^\infty(\H^n)
\end{equation}
provided $n\geq 5$ where $S_{n,k}$ denotes the sharp constant in the Sobolev inequality in Euclidean space $\R^n$
\begin{equation}\label{eq:SSRn}
\int_{\R^n} |\na^k u|^2 dx \geq S_{n,k}\lt(\int_{\R^n} |u|^{\frac{2n}{n-2k}} dx\rt)^{\frac{n-2k}{n}},\quad u\in C_0^\infty(\R^n)
\end{equation}
when $n\geqs 2k$. The constant $S_{n,1}$ was found out independently by Talenti \cite{Talenti1976} and Aubin \cite{Aubin1976}. The sharp constant $S_{n,k}, k\geq 2$ was computed explicitly by Lieb \cite{Lieb} by proving the sharp Hardy--Littlewood--Sobolev inequality which is the dual version of \eqref{eq:SSRn}. 
In \cite{Nguyen2020a} the author proved the following inequality: given $n\geq 4$ and $\frac{2n}{n-1}\leq q \leq p \leqs n$, then for any $q\leq l \leq \frac{nq}{n-p}$ we have
\begin{equation}\label{eq:PSLorentz1}
\|\na_g u\|_{p,q}^q - \lt(\frac{n-1}p\rt)^q \|u\|_{p,q}^q \geq  S_{n,p,q,l}^q \|u\|_{p^*,l}^q,\quad\forall\, u\in W^1 L^{p,q}(\H^n),
\end{equation}
where $p^* = \frac{np}{n-p}$, and
\[
S_{n,p,q,l} =\begin{cases}
\lt[n^{1-\frac ql} \sigma_n^{\frac qn} \Big(\frac{(n-p)(l-q)}{qp}\Big)^{q+\frac ql -1} S\Big(\frac{lq}{l-p},q\Big)\rt]^{\frac 1q} &\mbox{if $q \leqs l \leq \frac{nq}{n-p}$,}\\
\frac{n-p}p \sigma_n^{\frac 1n}&\mbox{if $l =q$,}
\end{cases}
\]
where $\si_n$ denotes the volume of unit ball in $\R^n$ and $S\big(\frac{lq}{l-p},q\Big)$ is the sharp constant in the Sobolev inequality with fractional dimension (see \cite{NguyenPLMS}). It is interesting that the constant $S_{n,p,q,l}$ in \eqref{eq:PSLorentz1} is sharp and coincides with the sharp constant in the Lorentz--Sobolev type inequality in Euclidean space $\R^n$,
\[
\|\na u\|_{p,q}^q \geq S_{n,p,q,l}^q \|u\|_{p^*,l}^q.
\]
The previous inequality was proved by Alvino \cite{Alvino1977} for $l=q \leq p \leqs n$ and by Cassani, Ruf and Tarsi \cite{CassaniRufTarsi2018} for $l = q \geq p$. Our next aim is to improve the inequality \eqref{eq:Poincarehigher} in spirit of \eqref{eq:NguyenPSorder2} and \eqref{eq:PSLorentz1}.

\begin{theorem}\label{MAINII}
Let $n \geqs m \geq 1$ be integers, $1\leqs p\leqs \frac nm$ and $\frac{2n}{n-1}\leq q \leqs \infty$. Suppose, in addition, that $q \leq p$ if $m$ is odd. Then there holds
\begin{equation}\label{eq:LorentzPS}
\|\na_g^m u\|_{p,q}^q - C(n,m,p)^q \|u\|_{p,q}^q \geq S(n,m,p)^q \|u\|_{p_m^*,q}^q,\quad u\in W^m L^{p,q}(\H^n),
\end{equation}
where $p_i^* = \frac{np}{n-ip}$, $i =0,1,2,\ldots,$ and
\[
S(n,m,p) = \begin{cases}
\si_n^{\frac mn} \prod_{i=0}^{k-1} \frac{n (n-2p_{2i}^*)}{p_{2i}^* (p_{2i}^*)'}  &\mbox{if $m = 2k$, $k\geq 1$}\\
\si_n^{\frac mn}\frac{n-p}p  \prod_{i=1}^{k} \frac{n (n-2p_{2i-1}^*)}{p_{2i-1}^* (p_{2i-1}^*)'}   &\mbox{if $m= 2k+1$, $k\geq 1$.}
\end{cases}
\]
\end{theorem}

The rest of this paper is organized as follows. In Section \S2, we recall some facts on the hyperbolic spaces and the non-increasing spherically symmetric rearrangement in the hyperbolic spaces. We also prepare some auxiliary results which are important in the proof of our main results. Theorem \ref{MAINI} is proved in Section \S3 while the Section \S4 is devoted to prove Theorem \ref{MAINII}.

\section{Preliminaries}
We start this section by briefly recalling some basis facts on the hyperbolic spaces and the Lorentz--Sobolev space defined in the hyperbolic spaces. Let $n\geq 2$, a hyperbolic space of dimension $n$ (denoted by $\H^n$) is a complete , simply connected Riemannian manifold having constant sectional curvature $-1$. There are several models for the hyperbolic space $\H^n$ such as the half-space model, the hyperboloid (or Lorentz) model and the Poincar\'e ball model. Notice that all these models are Riemannian isometry. In this paper, we are interested in the Poincar\'e ball model of the hyperbolic space since this model is very useful for questions involving rotational symmetry. In the Poincar\'e ball model, the hyperbolic space $\H^n$ is the open unit ball $B_n\subset \R^n$ equipped with the Riemannian metric
\[
g(x) = \Big(\frac2{1- |x|^2}\Big)^2 dx \otimes dx.
\]
The volume element of $\H^n$ with respect to the metric $g$ is given by
\[
dV_g(x) = \Big(\frac 2{1 -|x|^2}\Big)^n dx,
\]
where $dx$ is the usual Lebesgue measure in $\R^n$. For $x \in B_n$, let $d(0,x)$ denote the geodesic distance between $x$ and the origin, then we have $d(0,x) = \ln (1+|x|)/(1 -|x|)$. For $\rho \geqs 0$, $B(0,\rho)$ denote the geodesic ball with center at origin and radius $\rho$. If we denote by $\na$ and $\Delta$ the Euclidean gradient and Euclidean Laplacian, respectively as well as $\la \cdot, \cdot\ra$ the standard scalar product in $\R^n$, then the hyperbolic gradient $\na_g$ and the Laplace--Beltrami operator $\Delta_g$ in $\H^n$ with respect to metric $g$ are given by
\[
\na_g = \Big(\frac{1 -|x|^2}2\Big)^2 \na,\quad \Delta_g = \Big(\frac{1 -|x|^2}2\Big)^2 \Delta + (n-2) \Big(\frac{1 -|x|^2}2\Big)\la x, \na \ra,
\]
respectively. For a function $u$, we shall denote $\sqrt{g(\na_g u, \na_g u)}$ by $|\na_g u|_g$ for simplifying the notation. Finally, for a radial function $u$ (i.e., the function depends only on $d(0,x)$) we have the following polar coordinate formula
\begin{equation*}
\int_{\H^n} u(x) dx = n \sigma_n \int_0^\infty u(\rho) \sinh^{n-1}(\rho)\,  d\rho.
\end{equation*}

It is now known that the symmetrization argument works well in the setting of the hyperbolic. It is the key tool in the proof of several important inequalities such as the Poincar\'e inequality, the Sobolev inequality, the Moser--Trudinger inequality in $\H^n$. We shall see that this argument is also the key tool to establish the main results in the present paper. Let us recall some facts about the rearrangement argument in the hyperbolic space $\H^n$. A measurable function $u:\H^n \to \R$ is called vanishing at the infinity if for any $t >0$ the set $\{|u| > t\}$ has finite $V_g-$measure, i.e.,
\[
V_g(\{|u|> t\}) = \int_{\{|u|> t\}} dV_g < \infty.
\]
For such a function $u$, its distribution function is defined by
\[
\mu_u(t) = V_g( \{|u|> t\}).
\]
Notice that $t \to \mu_u(t)$ is non-increasing and right-continuous. The non-increasing rearrangement function $u^*$ of $u$ is defined by
\[
u^*(t) = \sup\{s > 0\, :\, \mu_u(s) > t\}.
\] 
The non-increasing, spherical symmetry, rearrangement function $u^\sharp$ of $u$ is defined by
\[
u^\sharp(x) = u^*(V_g(B(0,d(0,x)))),\quad x \in \H^n.
\]
It is well-known that $u$ and $u^\sharp$ have the same non-increasing rearrangement function (which is $u^*$). Finally, the maximal function $u^{**}$ of $u^*$ is defined by
\[
u^{**}(t) = \frac1t \int_0^t u^*(s) ds.
\]
Evidently, $u^*(t) \leq u^{**}(t)$.

For $1\leq p, q < \infty$, the Lorentz space $L^{p,q}(\H^n)$ is defined as the set of all measurable function $u: \H^n \to \R$ satisfying
\[
\|u\|_{L^{p,q}(\H^n)}: = \lt(\int_0^\infty \lt(t^{\frac1p} u^*(t)\rt)^q \frac{dt}t\rt)^{\frac1q} < \infty.
\]
It is clear that $L^{p,p}(\H^n) = L^p(\H^n)$. Moreover, the Lorentz spaces are monotone with respect to second exponent, namely
\[
L^{p,q_1}(\H^n) \subsetneq L^{p,q_2}(\H^n),\quad 1\leq q_1 < q_2 < \infty.
\]
The functional $ u\to \|u\|_{L^{p,q}(\H^n)}$ is not a norm in $L^{p,q}(\H^n)$ except the case $q \leq p$ (see \cite[Chapter $4$, Theorem $4.3$]{Bennett}). In general, it is a quasi-norm which  turns out to be equivalent to the norm obtained replacing $u^*$ by its maximal function $u^{**}$ in the definition of $\|\cdot\|_{L^{p,q}(\H^n)}$. Moreover, as a consequence of Hardy inequality, we have
\begin{proposition}
Given $p\in (1,\infty)$ and $q \in [1,\infty)$. Then for any function $u \in L^{p,q}(\H^n)$ it holds 
\begin{equation}\label{eq:Hardy}
\lt(\int_0^\infty \lt(t^{\frac1p} u^{**}(t)\rt)^q \frac{dt}t\rt)^{\frac1q} \leq \frac p{p-1} \lt(\int_0^\infty \lt(t^{\frac1p} u^*(t)\rt)^q \frac{dt}t\rt)^{\frac1q} = \frac p{p-1} \|u\|_{L^{p,q}(\H^n)}.
\end{equation}
\end{proposition}
For $1\leq p, q \leqs \infty$ and an integer $m\geq 1$, we define the $m-$th order Lorentz--Sobolev space $W^mL^{p,q}(\H^n)$ by taking the completion of $C_0^\infty(\H^n)$ under the quasi-norm
\[
\|\na_g^m u\|_{p,q} := \| |\na_g^m u|\|_{p,q}.
\]
It is obvious that $W^mL^{p,p}(\H^n) = W^{m,p}(\H^n)$ the $m-$th order Sobolev space in $\H^n$. In \cite{Nguyen2020a}, the author established the following P\'olya--Szeg\"o principle in the first order Lorenz--Sobolev spaces $W^1L^{p,q}(\H^n)$ which generalizes the classical P\'olya--Szeg\"o principle in the hyperbolic space.
\begin{theorem}
Let $n\geq 2$, $1\leq q \leq p \leqs \infty$ and $u\in W^{1}L^{p,q}(\H^n)$. Then $u^\sharp \in W^{1}L^{p,q}(\H^n)$ and 
$$\|\na_g u^\sharp\|_{p,q} \leq \|\na_g u\|_{p,q}.$$
\end{theorem}
For $r \geq 0$, define
\[
\Phi(r) = n \int_0^r \sinh^{n-1}(s) ds, \quad r\geq 0,
\]
and let $F$ be the function such that 
\[
r = n \si_n \int_0^{F(r)} \sinh^{n-1}(s) ds, \quad r\geq 0,
\]
i.e., $F(r) = \Phi^{-1}(r/\si_n)$. It was proved in \cite[Lemma $2.1$]{Nguyen2020a} that
\begin{equation}\label{eq:keyest}
\sinh^{q(n-1)}(F(t)) \geq \lt(\frac t{\si_n}\rt)^{q \frac{n-1}n} + \lt(\frac{n-1}n\rt)^q \lt(\frac t{\si_n}\rt)^q,\quad t \geq 0,
\end{equation}
provided $q \geq \frac{2n}{n-1}$. Moreover, we have the following result.
\begin{proposition}
Let $n \geq 2$. Then it holds
\begin{equation}\label{eq:keyyeu}
\sinh^{n}(F(t)) \geqs \frac t{\si_n},\quad t\geqs 0.
\end{equation}
\end{proposition}
\begin{proof}
Indeed, for $\rho \geqs 0$,we have
\[
n\int_0^\rho \sinh^{n-1}(s) ds \leqs n \int_0^\rho \sinh^{n-1}(s) \cosh(s) ds =  \sinh^{n}(\rho).
\]
Taking $\rho = F(t), t \geqs 0$ we obtain \eqref{eq:keyyeu}.
\end{proof}

\begin{proposition}\label{ddtang}
Let $n\geq 2$, then the function
\[
\vphi(t) =\frac{t}{\sinh^{n-1}(F(t))}
\]
is strictly increasing on $(0,\infty)$, and 
\[
\lim_{t\to \infty} \varphi(t) = \frac{n \si_n}{n-1}.
\]
\end{proposition}
\begin{proof}
Since $ t\mapsto F(t)$ is strictly increasing function, then it is enough to prove that the function
\[
\eta(\rho) = \frac{\int_0^\rho \sinh^{n-1}(s) ds}{\sinh^{n-1}(\rho)}
\]
is strictly increasing on $(0,\infty)$. Indeed, we have
\begin{align*}
\eta'(\rho) &= 1 -(n-1) \cosh(\rho) \frac{\int_0^\rho \sinh^{n-1}(s) ds}{\sinh^{n}(\rho)}\\
&=\frac1{\sinh^n (\rho)} \lt(\sinh^n(\rho) - (n-1)  \cosh(\rho) \int_0^\rho \sinh^{n-1}(s) ds\rt)\\
&=: \frac{\xi(\rho)}{\sinh^n (\rho)},
\end{align*}
and
\[
\xi'(\rho) =\cosh(\rho) \sinh^{n-1}(\rho) - (n-1) \sinh(\rho) \int_0^\rho \sinh^{n-1}(s) ds.
\]
For $\rho \geqs 0$, it holds
\[
(n-1)\int_0^\rho \sinh^{n-1}(s) ds \leqs (n-1) \int_0^\rho \sinh^{n-2}(s) \cosh(s) ds = \sinh^{n-1}(\rho),
\]
here we use $\cosh(s) \geqs \sinh(s)$ for $s \geqs 0$. Therefore, we get
\[
\xi'(\rho) \geqs \sinh^{n-1}(\rho) (\cosh(\rho) -\sinh(\rho)) \geqs 0,
\]
for $\rho \geqs 0$. Consequently, we have $\xi(\rho) \geqs \xi(0) =0$ for $\rho \geqs 0$. Hence, $\eta'(\rho) \geqs 0$ for $\rho \geqs 0$ which implies  that $\eta$ is strictly increasing function on $(0,\infty)$. By L'Hospital rule, we have 
\[
\lim_{\rho \to \infty} \eta(\rho) = \lim_{\rho\to \infty} \frac{\sinh^{n-1}(\rho)}{(n-1)\sinh^{n-2}(\rho) \cosh(\rho)} = \frac1{n-1}
\]
which yields the desired limit in this proposition.
\end{proof}

In the rest of this section, we shall frequently using the following one-dimensional Hardy inequality
\begin{lemma}
Let $1\leqs q \leqs p$. Then for any absolutely continuous function $u$ in $(0,\infty)$ such that $\lim_{t\to \infty} |u(t)| t^{\frac{p-q}q} =0$, it holds
\begin{equation}\label{eq:1DHardy}
\int_0^\infty |u'(t)|^q t^{p-1} dt \geq \lt(\frac{p-q}q\rt)^q \int_0^\infty |u(t)|^q t^{p-q-1} dt.
\end{equation}
\end{lemma}

Let $u \in C_0^\infty(\H^n)$ and $f = -\Delta_g u$. It was proved by Ngo and the author (see \cite[Proposition $2.2$]{NgoNguyenAMV}) that
\begin{equation}\label{eq:NgoNguyen}
u^*(t) \leq v(t):= \int_t^\infty \frac{s f^{**}(s)}{(n \si_n \sinh^{n-1}(F(s)))^2} ds,\quad t\geqs 0.
\end{equation}
The following results are important in the proof of Theorem \ref{MAINI} and Theorem \ref{MAINII}.
\begin{proposition}\label{order2I}
Let $n \geq 2$, $p\in (1,n)$ and $q \in (1,\infty)$. Then it holds
\begin{equation}\label{eq:order2I1}
\|\Delta_g u\|_{p,q}^q  \geq \lt(\frac{p-1}p n \si_n^{\frac1n}\rt)^q\int_0^\infty |v'(t)|^q (n\si_n \sinh^{n-1}(F(t)))^qt^{q(\frac1p -\frac1n) -1} dt.
\end{equation}
Furthermore, if $q \geq \frac{2n}{n-1}$ then we have
\begin{align}\label{eq:order2I2}
\|\Delta_g u\|_{p,q}^q  - &\lt(\frac{(n-1)(p-1)}p\rt)^q\int_0^\infty |v'(t)|^q (n\si_n \sinh^{n-1}(F(t)))^q t^{\frac qp -1} dt \notag\\
&\geq \lt(\frac{p-1}p n \si_n^{\frac1n}\rt)^q\int_0^\infty |v'(t)|^q (n\si_n \sinh^{n-1}(F(t)))^qt^{q(\frac1p -\frac1n) -1} dt.
\end{align}
\end{proposition}
\begin{proof}
We have 
\[
v'(t) =-\frac{t f^{**}(t)}{(n \si_n \sinh^{n-1}(F(t)))^2}, 
\]
and hence
\begin{equation}\label{eq:dthuc}
\int_0^\infty |v'(t)|^q (n\si_n \sinh^{n-1}(F(t)))^qt^{q(\frac1p -\frac1n) -1} dt =\int_0^{\infty}  \frac{(f^{**}(t))^q t^{q\frac{n-1}n}}{(n\si_n \sinh^{n-1}(F(t)))^q} t^{\frac{q}p -1} dt.
\end{equation}
Using \eqref{eq:keyyeu} and \eqref{eq:Hardy} we obtain
\begin{align*}
\int_0^\infty |v'(t)|^q (n\si_n \sinh^{n-1}(F(t)))^qt^{q(\frac1p -\frac1n) -1} dt &\leq \frac1{n^q \si_n^{\frac qn}} \int_0^\infty (f^{**}(t))^q  t^{\frac{q}p -1} dt\\
&\leq \frac1{n^q \si_n^{\frac qn}} \lt(\frac{p}{p-1}\rt)^q \int_0^\infty (f^{*}(t))^q  t^{\frac{q}p -1} dt\\
&=\lt(\frac1{n\si_n^{\frac1n}}\frac{p}{p-1}\rt)^q \|\Delta_g u\|_{p,q}^q,
\end{align*}
as wanted \eqref{eq:order2I1}.

We next prove \eqref{eq:order2I2}. We notice that
\[
\int_0^\infty |v'(t)|^q (n\si_n \sinh^{n-1}(F(t)))^q t^{\frac qp -1} dt =\int_0^{\infty}  \frac{(f^{**}(t))^q t^{q}}{(n\si_n \sinh^{n-1}(F(t)))^q} t^{\frac{q}p -1} dt
\]
This equality together with \eqref{eq:dthuc}, the fact $q \geq \frac{2n}{n-1}$ and the inequality \eqref{eq:Hardy} implies
\begin{align*}
&(n-1)^q\int_0^\infty |v'(t)|^q (n\si_n \sinh^{n-1}(F(t)))^q t^{\frac qp -1} dt\\
&\quad + n^q \si_n^{\frac qn}\int_0^\infty |v'(t)|^q (n\si_n \sinh^{n-1}(F(t)))^qt^{q(\frac1p -\frac1n) -1} dt\\
&\leq \int_0^\infty (f^{**}(t))^q \lt(\lt(\frac t{\si_n \sinh^n(F(t))}\rt)^{q \frac{n-1}n} + \lt(\frac{n-1}n\rt)^q \lt(\frac t{\si_n\sinh^{n-1}((F(t))}\rt)^q\rt) t^{\frac qp -1} dt\\
&\leq \int_0^\infty (f^{**}(t))^q t^{\frac qp -1} dt\\
&\leq \lt(\frac{p}{p-1}\rt)^q \int_0^\infty (f^{*}(t))^q t^{\frac qp -1} dt\\
&=\lt(\frac{p}{p-1}\rt)^q \|\Delta_g u\|_{p,q}^q
\end{align*}
as wanted \eqref{eq:order2I2}.

\end{proof}

\begin{proposition}\label{order2II}
Let $n\geq 2$, $p\in (1,n)$ and $q \geq \frac{2n}{n-1}$. Then we have
\begin{align}\label{eq:order2II2}
\int_0^\infty |v'(t)|^q (n\si_n \sinh^{n-1}(F(t)))^q t^{\frac qp -1} dt \geq& \lt(\frac{n-1}p\rt)^q \int_0^\infty |v(t)|^q t^{\frac qp -1} dt \notag\\
&+ n^q\si_n^{\frac qn} \int_0^\infty |v'(t)|^q t^{q(\frac1p -\frac1n) + q -1} dt,
\end{align}
and
\begin{align}\label{eq:order2II3}
\int_0^\infty |v'(t)|^q (n\si_n \sinh^{n-1}(F(t)))^q t^{q(\frac 1p -\frac1n)-1} dt \geq& \lt(\frac{(n-1)(n-p)}{np}\rt)^q \int_0^\infty |v(t)|^q t^{q(\frac 1p-\frac1n) -1} dt \notag\\
&+  n^q \si_n^{\frac qn} \int_0^\infty |v'(t)|^q t^{q(\frac1p -\frac2n) + q -1} dt,
\end{align}
\end{proposition}
\begin{proof}


If $q \geq \frac{2n}{n-1}$ then by using \eqref{eq:keyest} we have
\begin{align*}
\int_0^\infty |v'(t)|^q (n\si_n \sinh^{n-1}(F(t)))^q t^{\frac qp -1} dt \geq& (n-1)^q \int_0^\infty |v'(t)|^q t^{\frac qp +q -1} dt \notag\\
&+ n^q\si_n^{\frac qn} \int_0^\infty |v'(t)|^q t^{q(\frac1p -\frac1n) + q -1} dt,
\end{align*}
Using the one dimensional Hardy inequality \eqref{eq:1DHardy}, we have
\[
\int_0^\infty |v'(t)|^q t^{\frac qp + q -1} d t \geq \lt(\frac1 p\rt)^q \int_0^\infty |v(t)|^q t^{\frac qp -1}dt.
\]
Combining these two inequalities proves the inequality \eqref{eq:order2II2}.

Since $q \geq \frac{2n}{n-1}$ then by using again \eqref{eq:keyest}, we get
\begin{align*}
\int_0^\infty |v'(t)|^q (n\si_n \sinh^{n-1}(F(t)))^q t^{q(\frac1p-\frac1n) -1} dt \geq& (n-1)^q \int_0^\infty |v'(t)|^q t^{q(\frac 1p-\frac1n) +q -1} dt \notag\\
&+ n^q\si_n^{\frac qn} \int_0^\infty |v'(t)|^q t^{q(\frac1p -\frac2n) + q -1} dt.
\end{align*}
Using the one dimensional Hardy inequality \eqref{eq:1DHardy}, we have
\[
\int_0^\infty |v'(t)|^q t^{q(\frac 1p -\frac1n) + q -1} d t \geq \lt(\frac{n-p}{n p}\rt)^q \int_0^\infty |v(t)|^q t^{q(\frac 1p -\frac1n) -1}dt.
\]
Combining these two inequalities proves the inequality \eqref{eq:order2II3}.
\end{proof}

Combining Propositions \ref{order2I} and \ref{order2II}, we obtain
\begin{theorem}\label{keytheorem}
Let $n\geq 2$. If $p \in (1,n)$ and $q \in (1,\infty)$. For any $u \in C_0^\infty(\H^n)$ we define $v$ by \eqref{eq:NgoNguyen}. Then we have
\begin{equation}\label{eq:LSorder2*}
\|\Delta_g u\|_{p,q}^q \geq \lt(\frac{n^2 \si_n^{\frac2n}}{p'}\rt)^q \int_0^\infty |v'(t)|^q t^{q(\frac1p -\frac2n) + q -1} dt,
\end{equation}
where $p' = p/(p-1)$. In particular, if $p \in (1,\frac n2)$ then it holds
\begin{equation}\label{eq:LSorder2}
\|\Delta_g u\|_{p,q}^q \geq \lt(\frac{n(n-2p)}{p p'} \si_n^{\frac 2n}\rt)^q \|u\|_{p_2^*,q}^q.
\end{equation}
Furthermore, if $p\in (1,n)$ and $q \geq \frac{2n}{n-1}$ then we have
\begin{equation}\label{eq:improvedLS2}
\|\Delta_g u\|_{p,q}^q - C(n,2,p)^q \|u\|_{p,q}^q \geq \lt(\frac{n^2 \si_n^{\frac2n}}{p'}\rt)^q \int_0^\infty |v'(t)|^q t^{q(\frac1p -\frac2n) + q -1} dt.
\end{equation}
In particular, if $p\in (1,\frac n2)$ and $q \geq \frac{2n}{n-1}$ then we have
\begin{equation}\label{eq:improvedLS2a}
\|\Delta_g u\|_{p,q}^q - C(n,2,p)^q \|u\|_{p,q}^q \geq \lt(\frac{n(n-2p)}{p p'} \si_n^{\frac 2n}\rt)^q \|u\|_{p_2^*,q}^q,\quad u \in C_0^\infty(\H^n).
\end{equation}
\end{theorem}
It is worthing to mention here that in the Euclidean space $\R^n$, an analogue of the inequality \eqref{eq:LSorder2} was proved by Tarsi (see \cite[Theorem $2$]{Tarsi}).
\begin{proof}
Let $u \in C_0^\infty(\H^n)$ and $v$ be defined by \eqref{eq:NgoNguyen}. We know that $u^* \leq v$, then 
\begin{equation}\label{eq:major}
\|u\|_{p,q}^q \leq \int_0^\infty v(t)^q t^{\frac qp -1} dt,\quad\text{\rm and }\quad \|u\|_{p_2^*,q}^q \leq \int_0^\infty v(t)^q t^{\frac q{p_2^*} -1} dt.
\end{equation}
The inequality \eqref{eq:LSorder2*} is a consequence of \eqref{eq:order2I1} and \eqref{eq:keyest}. The inequality \eqref{eq:LSorder2} is consequence of \eqref{eq:LSorder2*}, the one dimensional Hardy inequality \eqref{eq:1DHardy}
\[
\int_0^\infty |v'(t)|^q t^{q(\frac1p -\frac 2n) + q -1} d t \geq \lt(\frac{n-2p}{np}\rt)^q \int_0^\infty |v(t)|^q t^{q(\frac1p -\frac2n) -1}dt.
\]
and the second inequality in \eqref{eq:major}. 

To prove \eqref{eq:improvedLS2}, we first notice by the first inequality in \eqref{eq:major} that
\[
\|\Delta_g u\|_{p,q}^q - C(n,2,p)^q \|u\|_{p,q}^q \geq \|\Delta_g u\|_{p,q}^q - C(n,2,p)^q \int_0^\infty v(t)^q t^{\frac qp -1} dt.
\]
Hence, it holds
\begin{align}\label{eq:cc}
&\|\Delta_g u\|_{p,q}^q - C(n,2,p)^q \|u\|_{p,q}^q\notag\\
\geq &\|\Delta_g u\|_{p,q}^q  - \lt(\frac{n-1}{p'}\rt)^q\int_0^\infty |v'(t)|^q (n\si_n \sinh^{n-1}(F(t)))^q t^{\frac qp -1} dt\notag\\
&+ \lt(\frac{n-1}{p'}\rt)^q\lt(\int_0^\infty |v'(t)|^q (n\si_n \sinh^{n-1}(F(t)))^q t^{\frac qp -1} dt - \lt(\frac{n-1}p\rt)^q \int_0^\infty v(t)^q t^{\frac qp -1} dt\rt)\notag\\
\geq&\|\Delta_g u\|_{p,q}^q  - \lt(\frac{n-1}{p'}\rt)^q\int_0^\infty |v'(t)|^q (n\si_n \sinh^{n-1}(F(t)))^q t^{\frac qp -1} dt,
\end{align}
here we use $q \geq \frac{2n}{n-1}$ and the inequality \eqref{eq:order2II2}. Using again the assumption $q \geq \frac{2n}{n-1}$ and the inequalities \eqref{eq:order2I2} and \eqref{eq:order2II3}, we obtain
\begin{align}\label{eq:cc1}
\|\Delta_g u\|_{p,q}^q  - \lt(\frac{n-1}{p'}\rt)^q&\int_0^\infty |v'(t)|^q (n\si_n \sinh^{n-1}(F(t)))^q t^{\frac qp -1} dt\notag \\
&\geq \lt(\frac{n^2 \si_n^{\frac2n}}{p'}\rt)^q \int_0^\infty |v'(t)|^q t^{q(\frac1p -\frac2n) + q -1} dt.
\end{align}
Combining the estimates \eqref{eq:cc} and \eqref{eq:cc1} proves \eqref{eq:improvedLS2}. The inequality \eqref{eq:improvedLS2a} follows from \eqref{eq:improvedLS2} and the one dimensional Hardy inequality \eqref{eq:1DHardy}
\[
\int_0^\infty |v'(t)|^q t^{q(\frac1p -\frac 2n) + q -1} d t \geq \lt(\frac{n-2p}{np}\rt)^q \int_0^\infty |v(t)|^q t^{q(\frac1p -\frac2n) -1}dt.
\]
\end{proof}

\section{Proof of Theorem \ref{MAINI}}
In this section, we prove Theorems \ref{MAINI}. 

\begin{proof}[Proof of Theorem \ref{MAINI}]
In the case $m =1$, Theorem \ref{MAINI} was already proved in \cite{Nguyen2020a}. So, we will only consider the case $m \geq 2$. We divide the proof into three cases as follows.

\emph{Case 1: $m=2$.} If $p \in (1,n)$ and $q\geq \frac {2n}{n-1}$ then \eqref{eq:Poincarehigher} follows from \eqref{eq:improvedLS2}. In the following, we will give a proof of \eqref{eq:Poincarehigher} for any $p, q \in (0,\infty)$. By the density, it is enough to prove \eqref{eq:Poincarehigher} for function $u\in C_0^\infty(\H^n)$, $u\not\equiv 0$. Let $f = -\Delta_g u$ and $v$ be defined by \eqref{eq:NgoNguyen}. We first notice that
\[
\int_0^\infty |v'(t)|^q t^{\frac qp + q -1} dt = \int_0^\infty \lt(\frac t{n\si_n \sinh^{n-1}(F(t))}\rt)^{2q} f^{**}(t)^q  t^{\frac qp-1} dt.
\]
By Lemma \ref{ddtang}, we have
\[
\frac t{n\si_n \sinh^{n-1}(F(t))} \leqs \frac1{n-1},\quad t > 0.
\]
This inequality together with the inequality \eqref{eq:Hardy} implies
\begin{align*}
\int_0^\infty \lt(\frac t{n\si_n \sinh^{n-1}(F(t))}\rt)^{2q} f^{**}(t)^q  t^{\frac qp-1} dt &\leq \lt(\frac{p'}{(n-1)^2}\rt)^q \int_0^\infty f^*(t)^q t^{\frac qp -1}dt \\
&= \lt(\frac{p'}{(n-1)^2}\rt)^q \|\Delta_g u\|_{p,q}^q.
\end{align*}
By the one dimensional Hardy inequality and the first inequality in \eqref{eq:major}, we have
\[
\int_0^\infty |v'(t)|^q t^{\frac qp + q -1} dt \geq \frac{1}{p^q} \int_0^\infty v(t)^q t^{\frac qp -1}dt \geq \frac{1}{p^q} \|u\|_{p,q}^q.
\]
Combining two previous inequalities, we obtain \eqref{eq:Poincarehigher}.

\emph{Case 2: $m =2k$, $k\geq 1$.} This case follows from the \emph{Case 1} and the iteration argument.

\emph{Case 3: $m= 2k+1$, $k\geq 1$.} Since $q \leq p$, then it was proved in \cite[Theorem $1.1$]{Nguyen2020a} that 
\[
\|\na_g \Delta_g^k u\|_{p,q}^q \geq \lt(\frac{n-1}p\rt)^q \|\Delta_g^k u\|_{p,q}^q.
\]
We now apply the \emph{Case 2} to obtain the desired result.

We next check the sharpness of the constant $C(n,m,p)$ in \eqref{eq:Poincarehigher}. From Proposition \ref{ddtang}, we see that for any $\ep >0$ there exists $a> 0$ such that 
\[
(n-1)s < n\si_n \sinh^{n-1}(F(s)) \leq (1+ \ep) (n-1)s,
\]
for any $s \geq a$. For $R > a$, let us define the function
\[
f_{R}(s) = \begin{cases}
a^{-\frac1p} &\mbox{if $s \in (0,a)$},\\
s^{-\frac1p} &\mbox{if $s\in [a,R)$,}\\
R^{-\frac1p} \max\{2 -s/R,0\} &\mbox{if $s \geq R$}.
\end{cases}
\]
Notice that $f_R$ is a nonnegative, continuous, non-increasing function. Following \cite[Section $2.2$]{NgoNguyenAMV}, we define two sequences of functions $\{v_{R,i}\}_{i\geq 0}$ and $\{g_{R,i}\}_{i\geq 1}$ as follows:
\begin{description}
\item (i) First, we set $v_{R,0} = f_R$,
\item (ii) then in terms of $v_{R,i}$, we define $g_{R,i+1}$ as the maximal function of $v_{R,i}$, i.e.
\[
g_{R,i+1}(t) = \frac1t \int_0^t v_{R,i}(s) ds,
\]
\item (iii) and finally in terms of $g_{R,i+1}$, we define $v_{R,i+1}$ as follows
\[
v_{R,i+1}(t) = \int_t^\infty \frac{s g_{R,i+1}(s)}{(n \si_n \sinh^{n-1}(F(s)))^2} ds,
\]
for $i= 0,1,2,\ldots$
\end{description}
Note that $v_{R,i}$ and $g_{R,i}$ are positive, non-increasing functions. Following the proof of \cite[Proposition $2.1$]{NgoNguyenAMV}, we can prove the following result.
\begin{proposition}
For any $i \geq 1$, there exist function $h_{R,i}$ and $w_{R,i}$ such that 
\[
v_{R,i} = h_{R,i} + w_{R,i},\quad \int_0^\infty |w_{R,i}|^q t^{\frac qp -1}dt \leq C
\]
and
\[
\frac{1}{(1+ \ep)^{2i}} \lt(\frac{pp'}{(n-1)^2}\rt)^i f_R \leq h_{R,i} \leq \lt(\frac{pp'}{(n-1)^2}\rt)^i,
\]
where we use $C$ to denote various constants which are independent of $R$.
\end{proposition}
\begin{proof}
Let us define the operator $T$ acting on functions $v$ on $[0,\infty)$ by
\[
T(v)(t) = \int_t^\infty \frac{s}{(n \si_n (\sinh (F(s)))^2}\lt(\frac1s\int_0^s v(r) dr\rt) ds.
\]
We shall prove that
\begin{equation}\label{eq:Tnorm}
\int_0^\infty |T(v)(t)|^q t^{\frac qp -1} dt \leq \lt(\frac{p p'}{(n-1)^2}\rt)^q \int_0^\infty |v(t)|^q t^{\frac qp -1} dt.
\end{equation}
Indeed, it is enough to prove \eqref{eq:Tnorm} for nonnegative function $v$ such that 
\[
\int_0^\infty |v(t)|^q t^{\frac qp -1} dt < \infty.
\]
We claim that 
\begin{equation}\label{eq:claimT}
\lim_{t\to 0} T(v)(t) t^{\frac1p} = 0 = \lim_{t\to \infty} T(v)(t) t^{\frac1p}.
\end{equation}
For any $\ep >0$ there exists $t_0 >0$ such that $\int_0^{t_0} |v(t)|^q t^{\frac qp -1} dt \leq \ep^q$. For $s \leq t_0$, by using H\"older inequality, we have
\[
\frac 1s \int_0^s v(r) dr \leq \frac1s \lt(\int_0^s v(r)^q r^{\frac qp -1} dr\rt)^{\frac1q} \lt(\int_0^s r^{\frac1{q-1} -\frac{q}{p(q-1)}} dr\rt)^{\frac{q-1}q} \leq C \ep s^{-\frac1p}.
\]
This together with the inequality $n \sigma_n \sinh^{n-1}(F(s)) > (n-1)s$ implies for $t\leq t_0$ that
\begin{align*}
T(v)(t) &=\Big(\int_t^{t_0} + \int_{t_0}^\infty\Big) \frac{s}{(n \si_n (\sinh (F(s)))^2}\lt(\frac1s\int_0^s v(r) dr\rt) ds\\
&\leq C \ep \int_t^{t_0} s^{-1-\frac1p} ds + C \lt(\int_0^\infty v(r)^q r^{\frac qp -1} dr\rt)^{\frac1q}\int_{t_0}^\infty s^{-1-\frac1p} ds\\
&\leq C \ep (t^{-\frac1p} -t_0^{-\frac1p}) + C t_0^{-\frac1p} \lt(\int_0^\infty v(r)^q r^{\frac qp -1} dr\rt)^{\frac1q}.
\end{align*}
This estimate yields
\[
\limsup_{t\to 0} T(v)(t) t^{\frac1p} \leq C \ep.
\]
Since $\ep >0$ is arbitrary, then the first limit in \eqref{eq:claimT} is proved.

Similarly, for any $\ep >0$ there exists $t_1 >0$ such that $\int_{t_1}^\infty |v(t)|^q t^{\frac qp -1} dt < \ep^q$. Hence, for $s \geq t_1$, by using H\"older inequality we get
\[
\int_0^s v(r) dr = \int_0^{t_1} v(r) dr + \int_{t_1}^s v(r) dr \leq C\lt(\int_{0}^\infty |v(t)|^q t^{\frac qp -1} dt\rt)^{\frac1q} t_1^{1-\frac1p}+ C \ep s^{1-\frac1p}.
\]
Consequently, for any $t\geq t_1$ we get
\begin{align*}
T(v)(t) &\leq C \int_t^\infty \lt(\lt(\int_{0}^\infty |v(t)|^q t^{\frac qp -1} dt\rt)^{\frac1q} t_1^{1-\frac1p} s^{-2} + \ep s^{-1-\frac1p}\rt) ds\\
&\leq C \lt(\int_{0}^\infty |v(t)|^q t^{\frac qp -1} dt\rt)^{\frac1q} t_1^{1-\frac1p} t^{-1} + C \ep t^{-\frac1p}.
\end{align*}
This estimate implies
\[
\limsup_{t\to \infty} T(v)(t) t^{\frac1p} \leq C \ep.
\]
Since $\ep >0$ is arbitrary, then the second limit in \eqref{eq:claimT} is proved.

Using the integration by parts, the  claim \eqref{eq:claimT} and the inequality $n\si_n \sinh^{n-1}(F(t)) > (n-1)t$, we have
\begin{align*}
\int_0^\infty T(v)(t)^q t^{\frac qp -1} dt &= \frac pq \int_0^\infty T(v)(t)^q (t^{\frac qp})' dt\\
&= p \int_0^\infty T(v)(t)^{q-1} \lt(\frac t{n \si_n \sinh^{n-1}(F(t))}\rt)^2\lt(\frac1t \int_0^t v(s) ds\rt) t^{\frac qp -1} dt\\
&\leq \frac{p}{(n-1)^2} \int_0^\infty T(v)(t)^{q-1} \lt(\frac1t \int_0^t v(s) ds\rt) t^{\frac qp -1} dt.
\end{align*}
An easy application of H\"older inequality implies
\[
\int_0^\infty T(v)(t)^q t^{\frac qp -1} dt \leq \lt(\frac p{(n-1)^2}\rt)^q \int_0^\infty \lt(\frac1t \int_0^t v(s) ds\rt)^q t^{\frac qp -1}dt.
\]
The inequality \eqref{eq:Tnorm} follows from the previous inequality and the Hardy inequality \eqref{eq:Hardy}.

Thus, with the help of \eqref{eq:Tnorm}, we can using the induction argument to prove this proposition by establishing the result for $v_{R,1}$. In fact, the decomposition for $v_{R,1}$ is already proved in the proof of Proposition $2.1$ in \cite{NgoNguyenAMV}. The estimate 
\[
\int_0^\infty |w_{R,1}|^q t^{\frac qp -1} dt \leq C,
\]
is proved by the same way of the estimate $\int_0^\infty |w_{R,1}|^p dt \leq C$.
\end{proof}

We are now ready to check the sharpness of $C(n,m,p)$. The case $m=1$ was done in \cite{Nguyen2020a}. Hence, we only consider the case $m\geq 2$. We first consider the case $m =2k, k\geq 1$. Define 
$$u_R(x) = v_{R,k}(V_g(B(0,d(0,x)))).$$
It is clear that $(-\Delta_g)^k u_R(x) = f_R(V_g(B(0,d(0,x))))$. Hence, there hold
\[
\|\Delta_g^k u_R\|_{p,q}^q =\int_0^\infty f_R(t)^q t^{\frac qp -1}dt = \frac pq + \ln \frac Ra + \int_1^2 (2-s)^q s^{\frac qp -1} ds,
\]
and
\begin{align*}
\|u_R\|_{p,q} &= \lt(\int_0^\infty v_{R,k}^q t^{\frac qp-1} dt\rt)^{\frac1q}\\
&\geq \lt(\int_0^\infty h_{R,k}^q t^{\frac qp-1} dt\rt)^{\frac1q} - \lt(\int_0^\infty |w_{R,k}|^q t^{\frac qp-1} dt\rt)^{\frac1q}\\
&\geq \frac1{(1+ \ep)^{2k})} \lt(\frac{pp'}{(n-1)^2}\rt)^k \lt(\int_0^\infty f_R^q t^{\frac qp-1} dt\rt)^{\frac1q} - C\\
&\geq \frac1{(1+ \ep)^{2k})} \lt(\frac{pp'}{(n-1)^2}\rt)^k \lt( \frac pq + \ln \frac Ra + \int_1^2 (2-s)^q s^{\frac qp -1} ds\rt)^{\frac1q} - C.
\end{align*}
These estimates imply
\[
\limsup_{R\to \infty} \frac{\|\Delta_g^k u_R\|_{p,q}^q}{\|u_R\|_{p,q}^q} \leq (1+ \ep)^{2kq} C(n,2k,p)^k,
\]
for any $\ep >0$. This proves the sharpness of $C(n,2k,p)$. We next consider the case $m =2k+1$, $k\geq 1$. Define
\[
u_R(x) = v_{R,k}(V_g(B(0,d(0,x)))).
\]
It is clear that $(-\Delta_g)^k u_R(x) = f_R(V_g(B(0,d(0,x))))$. It was shown in the proof of Theorem $1.1$ in \cite{Nguyen2020a} (the sharpness of $C(n,1,p)$) that 
\[
\limsup_{R\to \infty} \frac{\|\na_g \Delta_g^k u_R\|_{p,q}^q}{\|\Delta_g^k u_R\|_{p,q}^q} \leq (1+\ep)^q \frac{(n-1)^q}{p^q}.
\]
This together with the estimate in the case $m=2k$ implies
\begin{align*}
\limsup_{R\to \infty} \frac{\|\na_g \Delta_g^k u_R\|_{p,q}^q}{\|u_R\|_{p,q}^q} &\leq \limsup_{R\to \infty} \frac{\|\na_g \Delta_g^k u_R\|_{p,q}^q}{\|\Delta_g^k u_R\|_{p,q}^q} \limsup_{R\to\infty}\frac{\|\Delta_g^k u_R\|_{p,q}^q}{\|u_R\|_{p,q}^q}\\
& \leq (1+\ep)^{(2k+)q} C(n,2k+1,p)^q,
\end{align*}
for any $\ep \geqs 0$. This proves the sharpness of $C(n,2k+1,p)$. 

The proof of Theorem \ref{MAINI} is then completely finished.
\end{proof}

\section{Proof of Theorem \ref{MAINII}}
This section is addressed to prove Theorem \ref{MAINII}. The proof uses the results from Theorem \ref{keytheorem} and \cite[Theorem $1.2$]{Nguyen2020a}.

\begin{proof}[Proof of Theorem \ref{MAINII}]
The case $m=1$ was already proved in \cite{Nguyen2020a}. So, we only consider the case $m\geq 2$. We divide the proof into two cases as follows.

\emph{Case 1: $m=2k$, $k\geq 1$.} The case $k=1$ follows from \eqref{eq:improvedLS2a}. For $k\geq 2$, by using Theorem \ref{MAINI}, we have
\[
\|\Delta_g^k u\|_{p,q}^q - C(n,2k,p)^q \|u\|_{p,q}^q \geq \|\Delta_g^k u\|_{p,q}^q -C(n,2,p)^q \|\Delta^{k-1}_g u\|_{p,q}^q.
\]
Applying the inequality \eqref{eq:improvedLS2a} to the right hand side of the previous inequality, we obtain
\[
\|\Delta_g^k u\|_{p,q}^q - C(n,2k,p)^q \|u\|_{p,q}^q \geq \lt(\frac{n(n-2p)}{p p'} \si_n^{\frac 2n}\rt)^q \|\Delta_g^{k-1}u\|_{p_2^*,q}^q.
\]
By iterating the inequality \eqref{eq:LSorder2}, we then have
\[
\|\Delta_g^k u\|_{p,q}^q - C(n,2k,p)^q \|u\|_{p,q}^q \geq \lt(\si_n^{\frac {2k}n} \prod_{i=0}^{k-1} \frac{n (n-2p_{2i}^*)}{p_{2i}^* (p_{2i}^*)'}\rt)^q \|u\|_{p_{2k}^*,q}^q,
\]
as wanted \eqref{eq:LorentzPS}.

\emph{Case 1: $m=2k+1$, $k\geq 1$.} In this case, we have
\[
\|\na_g \Delta_g^k u\|_{p,q}^q - C(n,2k+1,p)^q \|u\|_{p,q}^q \geq \|\na_g \Delta_g^k u\|_{p,q}^q -\lt(\frac{n-1}p\rt)^q \|\Delta_g^k u\|_{p,q}^q.
\]
Since $q \leq p$ we then have from Theorem $1.2$ in \cite{Nguyen2020a} that
\[
\|\na_g \Delta_g^k u\|_{p,q}^q -\lt(\frac{n-1}p\rt)^q \|\Delta_g^k u\|_{p,q}^q \geq \lt(\frac{n-p}p \si_n^{\frac1n}\rt)^q \|\Delta_g^k u\|_{p^*,q}^q.
\]
Hence, it holds
\[
\|\na_g \Delta_g^k u\|_{p,q}^q - C(n,2k+1,p)^q \|u\|_{p,q}^q \geq \lt(\frac{n-p}p \si_n^{\frac1n}\rt)^q \|\Delta_g^k u\|_{p^*,q}^q.
\]
By iterating the inequality \eqref{eq:LSorder2}, we then have
\[
\|\na_g \Delta_g^k u\|_{p,q}^q - C(n,2k+1,p)^q \|u\|_{p,q}^q \geq \lt(\si_n^{\frac {2k+1}n}\frac{n-p}p  \prod_{i=1}^{k} \frac{n (n-2p_{2i-1}^*)}{p_{2i-1}^* (p_{2i-1}^*)'}\rt)^q \|u\|_{p_{2k+1}^*,q}^q,
\]
as desired \eqref{eq:LorentzPS}.
\end{proof}


\bibliographystyle{abbrv}

\end{document}